\documentclass[12pt, a4paper]{article}

\usepackage{amsmath}\multlinegap=0pt
\usepackage[latin1]{inputenc}
\usepackage{amsthm}
\usepackage{amssymb}
\usepackage[francais,english]{babel}
\numberwithin{equation}{section}
\theoremstyle{plain}
\newtheorem{thm}{Theorem}[section]

\newtheorem{prop}[thm]{Proposition}
\newtheorem{lem}[thm]{Lemma}
\newtheorem{defi}[thm]{Definition}

\theoremstyle{definition}

\theoremstyle{remark}
\newtheorem{rem}[thm]{Remark}

\newcommand{\R}{\mathbb{R}}

\newcommand\N{{\mathbb N}}

\newcommand\pref[1]{(\ref{#1})}

\let \eps\varepsilon

\newcommand\M{{\cal M}}

\DeclareMathOperator{\argmin}{argmin}

\DeclareMathOperator{\id}{id}

\newcommand\dive{\mathrm{div}}
\newcommand\supp{\mathrm{Supp}}

\def\<#1,#2>{\left<#1,#2\right>}

\newcommand\indi{{\mathbf{1}}}

\newcommand{\bn}{{\mbox{\boldmath$\nu$}}}
\newcommand{\bt}{{\mbox{\boldmath$\theta$}}}

\newcommand{\bxi}{{\mbox{\boldmath$\xi$}}}
\newcommand{\bg}{{\mbox{\boldmath$\gamma$}}}
\newcommand{\bv}{{\mbox{\boldmath$v$}}}
\newcommand{\bet}{{\mbox{\boldmath$\eta$}}}
\newcommand{\bw}{{\mbox{\boldmath$w$}}}
\newcommand{\bq}{{\mbox{\boldmath$q$}}}
\newcommand{\bid}{{\mbox{\boldmath$\id$}}}

\DeclareMathOperator{\sign}{sign}

\newcommand\hnu{{\hat{\nu}}}

\newcommand\bnu{{\overline \bn}}

\newcommand\ws{\stackrel{*}{\rightharpoonup}}
\def\PP{{\cal P}}

\title {Generalized solutions of a kinetic granular media equation by
a gradient  flow approach}
\author {Martial Agueh \thanks{\scriptsize  Department of Mathematics and statistics
University of Victoria, Victoria, BC, PO Box 3060 STN CSC
Victoria, BC, V8W 3R4, CANADA,  \texttt{agueh@math.uvic.ca}.} \;,\; Guillaume Carlier \thanks{\scriptsize CEREMADE, UMR CNRS 7534, Universit\'e Paris IX
Dauphine, Pl. de Lattre de Tassigny, 75775 Paris Cedex 16, FRANCE
\texttt{carlier@ceremade.dauphine.fr}.}}

\begin{document}

\maketitle

\begin{abstract}
We consider a one-dimensional kinetic model of granular media in the case where the interaction potential  is quadratic. Taking advantage of a simple first integral, we can use a reformulation (equivalent to the initial kinetic model for classical solutions) which allows measure solutions. This reformulation has a Wasserstein gradient flow structure (on a possibly infinite product of spaces of measures) for a convex energy which enables us to prove global in time well-posedness.
\end{abstract}

\textbf{Keywords:} Kinetic models of granular media, product Wasserstein space, gradient flows.\\

\textbf{AMS Subject Classifications:} 35Q70, 35D30, 35F25.


\section{Introduction}\label{sec-intro}

Kinetic models for granular media were initiated in the work of Benedetto, Caglioti and Pulvirenti \cite{bcp}, \cite{bcperr} who considered the following PDE
\begin{equation}\label{eqgene}
\partial_t f+ v\cdot \nabla_x f=\dive_v(f (\nabla W\star_v f)), \; (t,x,v)\in \R_+\times \R^d\times \R^d, \; f\vert_{t=0}=f_0,
\end{equation}
where $f_0$ is an integrable nonnegative function on the phase space and $W$ is a certain convex and radially symmetric potential capturing the (inelastic) collision rule between particles, and the convolution is in velocity only $(\nabla W\star_v f_t)(x,v)=\int_{\R^d} \nabla W(v-u)f_t(x,u) \mbox{d} u$   (so that there is no regularizing effect in the spatial variable). At least formally, \pref{eqgene} captures the limit as the number $N$ of particles tends to $+\infty$ of the second-order ODE system:
\begin{equation}\label{ode}
\dot{X}_i(t)=V_i(t), \; \dot{V}_i=-\frac{1}{N} \sum_{j\neq i} \nabla W(V_i(t)-V_j(t)) \delta_{X_i(t)-X_j(t)}, \; i=1, \cdots, N,
\end{equation} 
which describes the motion of $N$ particles of mass $\frac{1}{N}$ moving freely until collisions occur, and at collision times, there is some velocity  exchange  with a loss of kinetic energy depending on the form of the potential $W$. 


Surprisingly there are very few results on well-posedness for such equations. This is in contrast with the spatially homogeneous case (i.e. $f$ depending on $t$ and $v$ only) associated with (\ref{eqgene}) that has been very much studied (see  \cite{bcp,CMV1,CMV2,Laurent,BLR,CDFL} and the references therein) and  for which existence, uniqueness and long-time behavior are well understood. In fact, the spatially homogeneous version of (\ref{eqgene}) can be seen as the Wasserstein gradient flow of the interaction  energy associated to $W$, and then well-posedness results can be viewed as a consequence of the powerful theory of Wasserstein gradient flows (see \cite{AGS}).  For the full kinetic equation \pref{eqgene}, local existence and uniqueness of a classical solution was proved in one dimension in \cite{bcp} for the potential $W(v)=|v|^3/3\,$ (as observed in \cite{aci}, the arguments of \cite{bcp} extend to dimension $d$ and $W(v)=|v|^p/p$ provided $p>3-d$) when the initial datum $f_0$ is a non-negative $C^1\cap W^{1,\infty}(\R\times\R)$ integrable function  with compact support. Under an additional smallness assumption, the authors of \cite{bcp} also proved a global existence result. In \cite{agueh-kin}, the first author has extended the local existence result of \cite{bcp} to more general interaction potentials $W$ and to any dimension, $d\geq 1$.  The proof of \cite{agueh-kin} is based on a splitting of the kinetic equation (\ref{eqgene}) into a free transport equation in $x$, and a collision equation in $v$ that is interpreted as the gradient flow of a convex interaction energy with respect to the quadratic Wasserstein distance. In \cite{aci}, various a priori estimates are obtained, in particular a global entropy bound (which thus rules out concentration in finite time) in dimension $1$ when $W''$ is subquadratic near zero.

\smallskip

Understanding  under which conditions one can hope for global existence or on the contrary expect explosion in finite time is mainly an open question. Let us remark that the weak formulation of \pref{eqgene} means that for any $T>0$ and any $\phi\in C_c^{\infty}([0,T]\times \R^d\times \R^d)$ one has
\[\begin{split}
\int_0^T \int_{\R^d\times \R^d} ( \partial_t \phi(t,x,v) f_t(x,v) + \nabla_x \phi(t,x,v)\cdot v f_t(x,v) ) \mbox{d} x \mbox{d} v \mbox{d} t \\
=\int_{\R^d\times \R^d}  \phi(T,x,v) f_T(x,v) \mbox{d} x \mbox{d} v-\int_{\R^d\times \R^d}  \phi(0,x,v) f_0(x,v) \mbox{d} x \mbox{d} v\\
+ \int_0^T \int_{\R^d\times \R^d\times \R^d} \nabla_v \phi(t,x,v)\cdot \nabla W(v-u)  f_t(x,v) f_t(x,u) \mbox{d} x \mbox{d} u  \mbox{d} v \mbox{d} t 
\end{split}\]
and for the right hand side to make sense, it is necessary to have a control on nonlinear quantities like 
\[ \int_0^T \int_{\R^d\times \R^d\times \R^d}   f_t(x,v) f_t(x,u) \mbox{d} x \mbox{d} u \mbox{d} v \mbox{d} t\]
 which actually makes it difficult to define measure solutions (this also explains why in \cite{bcp} or \cite{agueh-kin}, the authors look for $L^1\cap L^\infty$ solutions). Observing that \pref{eqgene} can be written in conservative form as
 \[\partial_t f + \dive_{x,v} (f F(f))=0, \mbox{ with } F(f)(x,v)=(v, -(\nabla W\star_v f)(x,v)),\]
 we see that, at least for smooth solutions, \pref{eqgene} can be integrated using the method of characteristics:
 \[f_t={S_t}_\# f_0\]
 where $S_t$ is the flow of the vector-field $F(f)$ i.e. 
 \[S_0(x,v)=(x,v), \; \frac{d} {dt} S_t(x,v)=F(f_t)(S_t(x,v)),\]
and $f_t={S_t}_\# f_0$ means that
 \[\int_{\R^d\times \R^d}\varphi(x,v) f_t(x,v)\mbox{d} x \mbox{d} v  = \int_{\R^d\times \R^d}\varphi(S_t(x,v)) f_0(x,v)\mbox{d} x \mbox{d} v, \; \forall \varphi\in C_b(\R^d\times\R^d).\]

  In the present work, we investigate the one-dimensional case with the quadratic kernel $W(v)=\frac{1}{2} \vert v\vert^2$ which is neither covered by the analysis of \cite{bcp} nor by the entropy estimate of \cite{aci} (actually the entropy cannot be globally bounded in this case, see \cite{aci}). In this case the convolution  takes the form
\[ \int_{\R}(v-u)f_t(x,u) \mbox{d} u=\rho_t(x) v -m_t(x),\]
where 
\begin{equation}\label{rhom}
\rho_t(x):=\int_{\R} f_t(x,v) \mbox{d} v, \; m_t(x):=\int_{\R} v f_t(x,v) \mbox{d} v,
\end{equation}
so that the kinetic equation \pref{eqgene} rewrites
\begin{equation}\label{kinquad}
\partial_t f_t(x,v)+v \partial_x f_t(x,v)= \partial_v\Big(f_t(x,v)(\rho_t(x) v -m_t(x)) \Big),
\end{equation}
and we supplement \pref{kinquad} with the initial condition 
\begin{equation}\label{initcond}
f\vert_{t=0}=f_0,
\end{equation}
where $f_0$ is a compactly supported probability density:
 \begin{equation}
 f_0\in L^1(\R^d\times \R^d), \; \int_{\R^d\times \R^d} f_0 \mbox{d} x \mbox{d} v=1
 \end{equation}
and
\begin{equation}\label{compsupp0}
\supp(f_0) \subset B_{R_x} \times B_{R_v}
\end{equation}
for some positive constants $R_x$ and $R_v$. We shall see later on, how to treat more general measures as initial conditions. Our first contribution is the observation that, thanks to a special first integral of motion for the characteristics system associated with \pref{kinquad}, one may define weak solutions not at the level of measures on the phase space but on a (possibly infinite) product of measures on the physical space. Our second contribution is to  show that this reformulation has a gradient flow structure for an energy functional with good properties which will enable us to prove global well-posedness. To the best of our knowledge, even if the situation we are dealing with is very particular, this is the first global result of this type for kinetic models of granular media. As pointed out to us by Yann Brenier, our analysis has some similarities with (but is different from) some models of sticky particles for pressureless flows (see \cite{bregre}, \cite{bgsw}) and Brenier's formulation of the Darcy-Boussinesq system \cite{brenierdb}.

The article is organized as follows. In section \ref{sec-firstint}, we show how a certain first integral  of motion can be used to  give a reformulation of \pref{kinquad} which allows for measure solutions.  Section \ref{sec-gf} investigates the gradient flow structure of this reformulation. Section \ref{sec-jko} proves global existence thanks to the celebrated Jordan-Kinderlehrer-Otto (henceforth JKO) implicit Euler scheme of \cite{jko} for a certain energy functional. In section \ref{sec-wp}, we prove uniqueness and stability and give some concluding remarks.

\section{A first integral and measure solutions}\label{sec-firstint}

\subsection{A first integral for classical solutions}

Let us consider a $C^1$ compactly supported initial condition $f_0$ and a classical solution $f$, that is a $C^1$ function which solves \pref{kinquad} in a pointwise sense on $\R_+\times \R^d\times \R^d$. It is then easy to show (see \cite{aci}) that $f$ remains compactly supported locally in time; more precisely \pref{compsupp0} and \pref{kinquad} imply that
\begin{equation}\label{compsuppt}
\supp(f_t)\subset B_{R_x+t R_v}\times B_{R_v}, \; \forall t\ge 0.
\end{equation}
The characteristics for \pref{kinquad} is the flow map for the second-order ODE
\begin{equation}\label{secondorderODE}
\ddot{X}= -\rho_t(X) \dot{X}+m_t(X)
\end{equation}
in the sense that
\[f_t =(X_t, V_t)_\# f_0\]
where $(X_0(x,v), V_0(x,v))=(x,v)$ and 
\begin{equation}\label{char}
\frac{d}{dt} X_t(x,v)=V_t(x,v), \; \frac{d}{dt} V_t(x,v)=-\rho_t(X_t(x,v)) V_t(x,v)+m_t(X_t(x,v)),\;
\end{equation}
with $\rho$ and $m$ being respectively the spatial marginal and momentum associated to $f$ defined by \pref{rhom}. Integrating \pref{kinquad} with respect to $v$, first gives:
\begin{equation}\label{divfree}
\partial_t \rho_t(x)+\partial_x m_t(x)=0, \; t\ge 0, \; x\in \R
\end{equation}
so that there is a stream potential $G$ such that
\begin{equation}
\rho=\partial_x G, \; m=-\partial_t G ,
\end{equation}
and since $\rho$ is a probability measure, it is natural to choose the integration constant in such a way that $G$ is the cumulative distribution function of $\rho$:
\begin{equation}\label{new1}
G_t(x)=\int_{-\infty}^x \rho_t(y) \mbox{d} y=\rho_t((-\infty, x]).
\end{equation}
Replacing (\ref{new1}) in \pref{secondorderODE} then gives
\[\ddot{X}= -\partial_xG_t(X) \dot{X}-\partial_t G_t(X)=-\frac{d} {dt} G_t(X) \] 
so that  $\dot{X}+G_t(X)$ is constant along the characteristics. Since $G_0$ can be deduced  from the initial condition $f_0$  by
\[G_0(x)=\int_{-\infty}^x \int_\R f_0(y,v) \mbox{d}v \; \mbox{d} y,\]
we have the following explicit first integral of motion for \pref{char}:
\begin{equation}\label{firstinteg}
V_t(x,v)+G_t( X_t(x,v))=v+G_0(x).
\end{equation}

\subsection{Reformulation and equivalence for classical solutions}\label{classic} 

In view of the first integral \pref{firstinteg}, it is natural to perform a change of variables on the initial conditions:
\[a(x,v):=v+G_0(x)), \;  \nu_0^a(x):=f_0(x,a-G_0(x))\]
so that for every $\phi\in C(\R\times \R)$ one has
\[\int_{\R\times \R} \phi(x, a(x,v)) f_0(x,v) \mbox{d} x  \mbox{d} v= \int_{\R\times \R} \phi(x,a)  \nu_0^a(x) \mbox{d} x  \mbox{d} a,\]
and then to rewrite the characteristics as a family of first-order ODEs parametrized by the label $a$:
  \begin{equation}\label{paramflow}
\frac{d}{dt} X_t^a(x)=a-G_t(X_t^a(x)),\; X_0^a (x)=x.
\end{equation}
The flow \pref{char} may then be rewritten as:
\[X_t(x,v)=X_t^a (x), \; V_t(x,v)=a-G_t(X_t^a(x)) \mbox{ for } a=a(x,v)=v+G_0(x).\]
Hence setting
\begin{equation}
\nu_t^a:={X_t^a}_\# \nu_0^a,
\end{equation}
the relation $f_t =(X_t, V_t)_\# f_0$ can be re-expressed as:
\begin{equation}\label{recoverf}
\int_{\R^2} \phi(x,v)f_t(x, v) \mbox{d}x \mbox{d} v=\int_{\R^2} \phi(x, a-G_t(x)) \nu_t^a(x) \mbox{d}x \mbox{d} a
\end{equation}
for every $t\ge 0$ and every test-function $\phi\in C(\R^2)$. This implies in particular that
\[\rho_t(x)=\int_\R \nu_t^a(x) \mbox{d}a\]
and then also 
\begin{equation}\label{coupling}
G_t(x)=\int_\R G_t^a(x) \mbox{d}a \; \mbox{  with }  \; G_t^a (x):=\nu_t^a((-\infty,x]).
\end{equation}

On the other hand, using \pref{paramflow}, we deduce that for each $a\in \R$, $\nu^a$ satisfies the continuity equation:
\begin{equation}\label{contnu}
\partial_t \nu_t^a +\partial_x \Big( \nu_t ^a (a-G_t(x))\Big)=0, \; \nu^a \vert_{t=0}(x)=\nu_0^a(x)=f_0(x, a -G_0(x)).
\end{equation}
Note that $\nu_t^a$ is a nonnegative measure but not necessarily a probability measure, its total mass being that of $\nu_0^a$ i.e. $h(a):=\int_\R f_0(x, a-G_0(x)) \mbox{d} x$.  

\smallskip

The previous considerations show that any classical solution of \pref{kinquad} is related to a solution of the system of continuity equations \pref{contnu}-\pref{coupling} with initial condition $f_0$ via the relation \pref{recoverf}. The converse is also true: if $\nu^a$ is a family of classical solutions of \pref{contnu} with $G^a$ and $G$ given by \pref{coupling}, then the time-dependent family of probability measures $f_t$ on $\R^2$ defined by \pref{recoverf} actually solves \pref{kinquad}. Indeed, by construction the spatial marginal $\rho$ of $f$  is $\partial_x G$; as for the momentum, we have
\[m_t(x):=\int_{\R} v f_t(x,v) \mbox{d}v =\int_{\R} (a-G_t(x)) \nu_t^a (x)  \mbox{d}a.\] 
Then,  thanks to \pref{contnu} and Fubini's theorem, we have
\[\begin{split}
\partial_t G(x)=\int_{-\infty}^x \int_\R \partial_t \nu^a(y) \mbox{d}y \mbox{d}a
=- \int_{-\infty}^x \int_\R  \partial_x(\nu_t^a(y) (a-G_t(y)))   \mbox{d}y \mbox{d}a\\=-\int_{\R} (a-G_t(x)) \nu_t^a (x) dx=-m_t(x).
\end{split}\]
Then  let us take a test-function $\phi\in C_c^1(\R^2)$, differentiating \pref{recoverf} with respect to time, using  $\partial_x G=\rho$, $\partial_t G=-m$, \pref{recoverf} and an integration by parts and \pref{contnu}, we have
\[\begin{split}
\frac{d}{dt} \int_{\R^2} \phi f_t&=\int_{\R^2} \Big( -\phi(x, a-G_t(x))\partial_x(\nu_t^a (a-G_t))+\partial_v \phi(x, a-G_t) m_t \nu_t^a  \Big)  \mbox{d}x \mbox{d}a\\
&=\int_{\R^2} \Big( \partial_x \phi(x, a-G_t(x))   -\partial_v \phi(x, a-G_t(x)) \rho_t(x) \Big)   (a-G_t(x)) \nu_t^a(x)  \mbox{d}x \mbox{d}a\\
&+\int_{\R^2} \partial_v \phi(x, v) m_t(x) f_t(x,v)  \mbox{d}x \mbox{d}v\\
&= \int_{\R^2}( \partial_x \phi(x, v) v+ \partial_v \phi(x, v) (m_t(x)-\rho_t(x)v) f_t(x,v))\mbox{d}x \mbox{d}v.
\end{split}\]

This proves that, for classical solutions, the kinetic equation \pref{kinquad} is actually equivalent to the system of PDEs \pref{contnu}-\pref{coupling} indexed by the label $a$.  

\subsection{Measure solutions}

We now take the system \pref{contnu}-\pref{coupling} as a starting point to define measure solutions. We have to suitably relax the system so as to take into account:

\begin{itemize}

\item the fact that shocks may occur i.e.  atoms of $\rho$ may appear in finite time, then the cumulative distribution $G$ may become discontinuous (in which case it will be convenient to view $G$, which is monotone, as a set-valued map),

\item the fact that when shocks occur, the velocity may depend on the label $a$,

\item more general initial conditions.

\end{itemize}

Let us treat first the case of more general initial conditions. What really matters is to be able to make the change of variables $a=v+G_0(x)$ in a non-ambiguous way, which can be done as soon as $\rho_0$ is atomless i.e. does not charge points. We shall therefore assume that $f_0$ is a probability measure on $\R^2$ with compact support and having an atomless spatial marginal:
\begin{equation}\label{hypf0}
\supp(f_0) \subset B_{R_x} \times B_{R_v}, \; \rho_0 \mbox{ is atomless i.e. } f_0(\{x\}\times \R)=0, \; \forall x\in \R.
\end{equation}
Defining the spatial marginal $\rho_0$ of $f_0$ by
\[\int_\R \phi(x) \mbox{d}\rho_0(x)=\int_{\R^2} \phi(x) \mbox{d}f_0(x,v), \; \forall \phi \in C(\R)\]
as well as its cumulative distribution function
\[G_0(x):=\rho_0((-\infty,x])=f_0((-\infty,x]\times \R), \; \forall x\in \R,\]
$G_0$ is continuous and $\rho_0$ is suppported on $[-R_x, R_x]$. Since $G_0$ takes values in $[0,1]$, then $a(x,v):=v+G_0(x)\in [-R_v, R_v+1]$ for $(x,v)\in \supp(f_0)$. We then define the probability measure $\eta_0$ as the push-forward of $f_0$ through $(x,v)\mapsto (x, a(x,v))$ i.e.
\begin{equation}
\eta_0(C):=f_0 \Big(\{ (x,v) \; : \; (x,v+G_0(x))\in C\}\Big), \mbox{ for every Borel subset  $C$ of $\R^2$}.
\end{equation}
We then fix a $\sigma$-finite measure $\mu$  such that the second marginal of $\eta_0$ is absolutely continuous with respect to $\mu$; for instance it could be the second marginal of $\eta_0$, but we allow $\mu$ to be a more general measure (not necessarily a probability measure; for instance it was the Lebesgue measure in the previous paragraph \ref{classic}, and in the discrete example of paragraph \ref{sec-burg} below, $\mu$ will be a discrete measure). Then we can disintegrate $\eta_0$ as $\eta_0=\nu_0^a \otimes \mu$ which means that  for every $\phi \in C(\R^2)$ we have
\[\int_{\R^2} \phi(x, v+G_0(x)) \mbox{d} f_0(x,v)=\int_{\R} \Big( \int_\R \phi(x,a) \mbox{d} \nu_0^a(x)\Big) \mbox{ d} \mu(a).\]
Note that $\nu_0^a$ is supported on $[-R_x, R_x]$ and it is not necessarily a probability measure. We denote by $h(a)$ its total mass i.e. the Radon-Nikodym density of the second marginal of $\eta_0$ with respect to $\mu$:
\begin{equation}\label{defdeh}
 \int_{\R^2} \phi(v+G_0(x)) \mbox{d} f_0(x,v)=\int_{\R}  \phi(a) h(a) \mbox{ d} \mu(a), \; \forall \phi \in C(\R)
 \end{equation}
 so that $h\in L^1(\mu)$, $\int_{\R} h(a) \mbox{ d} \mu(a)=1$ and $h=0$ outside of the interval $[-R_v, R_v+1]$.
 
 \smallskip
 
 The rest of the paper will be devoted to study the structure and  well-posedness of the following system which relaxes to a measure-valued setting the system \pref{contnu}-\pref{coupling}:
\begin{equation}\label{s1}
\partial_t \nu_t^a +\partial_x (\nu_t^a v_t^a)=0, \; \nu^a\vert_{t=0} =\nu_0^a,
\end{equation} 
 subject to the constraint that
 \begin{equation}\label{s2}
  v_t^a(x) \in [a-G_t(x), a-G_t^{-}(x)]
 \end{equation} 
 where
 \begin{equation}\label{s3}
  \rho_t :=\int_{\R} \nu_t^a \mbox{d}\mu(a), G_t(x)=\rho_t((-\infty, x]), \;  G_t^-(x)=\rho_t((-\infty, x)).
 \end{equation}
 
Note that when $\mu$ is the Lebesgue measure and there are no shocks i.e. when $G_t$ is continuous, we recover the system \pref{contnu}-\pref{coupling} of paragraph \ref{classic}. Denoting by $\PP_2(\R)$ the set of Borel probability measures on $\R$ with finite second moment,  solutions of \pref{s1}-\pref{s2}-\pref{s3} are then formally defined by:

\begin{defi}
Fix a time $T>0$; a measure solution of the system \pref{s1}-\pref{s2}-\pref{s3}  on $[0,T]\times \R$ is a family of measures $(t,a) \in [0,T]\times [-R_v, R_v+1]  \mapsto \nu_t^a \in h(a) \PP_2(\R)$ which

\begin{enumerate}

\item is measurable in the sense that for every Borel bounded function $\phi$ on  $[0,T]\times \R\times \R$, the map $(t,a)\mapsto \int_\R \phi(t,a,x)\mbox{d} \nu_t^a(x)$  is $ \mbox{d}t \otimes \mu$ measurable,

\item satisfies the continuity equation \pref{s1} in the sense of distributions for $h\mu$-a.e. $a$, with a  $\nu_t^a\otimes \mu \otimes \mbox{d}t$-measurable velocity field $v_t^a$ which satisfies \pref{s2}, $\nu_t^a\otimes \mu \otimes \mbox{d}t$ a.e, and with $G_t$ and $G_t^{-}$ defined by \pref{s3}.

\end{enumerate}

\end{defi}

Note that since $v_t^a$ constrained by \pref{s3} is bounded, $t\mapsto \nu_t^a$ is actually continuous for the weak convergence of measures for $h\mu$ a.e. $a$. Note also that the fact that $t\mapsto \nu_t^a$ satisfies the continuity equation \pref{s1} in the sense of distributions is equivalent to the condition that for every $\psi\in C( [-R_v, R_v+1])$ and $\phi\in C_c^1([0,T]\times \R)$ one has:
\[\begin{split}
\int_{\R} \psi(a) \Big(\int_0^T \int_{\R} (\partial_t \phi(t,x)+\partial_x \phi(t,x) v_t^a(x) ) \mbox{d} \nu_t^a(x) \mbox{d} t\Big) \mbox{d} \mu(a)\\
= \int_{\R} \psi(a) \Big( \int_{\R} \phi(T,x) \mbox{d} \nu_T^a(x)-  \int_{\R} \phi(0,x) \mbox{d} \nu_0^a(x)  \Big) \mbox{d} \mu(a).
\end{split}\]

\subsection{A discrete example and a system of Burgers equations}\label{sec-burg}

As an example, let us consider the special case 
\[f_0=\rho_0 \otimes \frac{1}{N} \sum_{i=1}^N \delta_{a_i-G_0(x)}\]
where $\rho_0$ is a smooth compactly supported probability density and $a_1<\cdots <a_N$ are the finitely many values that the label $a$ may take. In this case, we take $\mu$ as the counting measure and then
\[\mu= \sum_{i=1}^N \delta_{a_i}, \; h(a_i)=\frac{1}{N}, \; \nu_0^{a_i}=\frac{1}{N} \rho_0.\]
Even though $G_0$ is smooth, we have to expect that shocks may appear in finite time. Let us relabel the measures $\nu^i:=\nu^{a_i}$ and the corresponding cumulative distributions $G^i:=G^{a_i}$, $G:=\sum_{j=1}^N G^j$. If there were no shocks, the system   \pref{s1}-\pref{s2}-\pref{s3}  would become 
\begin{equation}\label{discontsy}
\partial_ t \nu^i +\partial_x (\nu^i(a_i-\sum_{j=1}^N G_j))=0, \; \nu^i\vert_{t=0}=\frac{1}{N} \rho_0, \; i=1, \cdots, N.
\end{equation}
Integrating with respect to the spatial variable between $-\infty$ and $x$ would then give a system of Burgers-like equations:
\begin{equation}\label{burgers}
\partial_t G^i +\partial_x G^i (a_i-\sum_{j=1}^N G^j)=0, \; G^i\vert_{t=0}=\frac{1}{N} G_0, \; i=1, \cdots, N.
\end{equation}
We can at least formally rewrite each of these equations in the more familiar form
\[ \partial_t G^i +\partial_x G^i \psi^i_t(G^i) =0\]
where each function $\psi^i$ is implicitly defined in terms of the pseudo inverse $H^i_t$ of $G^i_t$:
\[\psi^i_t (\alpha)=a_i-\alpha-\sum_{j\neq i} G^j_t (H^i_t(\alpha)).\]
Note that $\psi_t^i$ is decreasing for every $t$ and actually $(\psi_t^i)'\le -1$. In the absence of shocks,  $H_t^i$ simply solves $\partial_t H^i=\psi^i_t$. Let us then  take $x_1<x_2$ belonging to a certain interval on which $\rho_0\ge \nu$ with $\nu>0$ and  define $y_1:=\frac{1}{N} G_0(x_1)$, $y_2:=\frac{1}{N} G_0(x_2)$, we then have $y_2-y_1=\frac{1}{N} \int_{x_1}^{x_2} \rho_0\ge \frac{\nu}{N}(x_2-x_1)$. Integrating $\partial_t H^i=\psi^i_t$ and using the fact that $(\psi^i)'\le -1$, we get
\[H^i_t( y_2)-H^i_t(y_1)=x_2-x_1 +\int_0^t (\psi^i_s(y_2)-\psi^i_s(y_1)) \mbox{d} s\le x_2-x_1-t(y_2-y_1).\]
This means that $H^i_t$ becomes noninjective before a time
\[\frac{x_2-x_1}{y_2-y_1}\le \frac{N}{\nu}.\]
In other words, discontinuities of $G^i$ i.e. shocks appear in finite time $O(N)$. 


\section{A gradient flow structure}\label{sec-gf}

In this section, assuming \pref{hypf0} we will see how to obtain solutions to the system \pref{s1}-\pref{s2}-\pref{s3} by a gradient flow approach. Existence of such gradient flows using the JKO implicit scheme for Wasserstein gradient flows will be detailed in section \ref{sec-jko}. We denote by $\M(\R^d)$ the set of Borel measures on $\R^d$ and $\PP(\R^d)$ the set of Borel probability measures on $\R^d$. Given two nonnegative Borel measures on $\R^d$ with common finite total mass $h$ (not necessarily $1$) and finite $p$-moments, $\nu$ and $\theta$, recall that for $p\in[1,+\infty)$, the $p$-Wasserstein distance between $\nu$ and $\theta$ is by definition:
\[W_p(\nu, \theta):= \inf_{\gamma\in \Pi(\nu, \theta) }  \Big\{\int_{\R^d\times \R^d} \vert x-y\vert^p  \mbox{d} \gamma(x,y) \Big\}^{\frac{1}{p}} \]
where $ \Pi(\nu, \theta)$ is the set of transport plans between $\nu$ and $\theta$ i.e. the set of  Borel probability measures on $\R^d\times \R^d$ having $\nu$ and $\theta$ as marginals (we refer to the textbooks of Villani \cite{villani, villani2} for a detailed exposition of optimal transport theory). Wasserstein distances are usually defined between probability measures  such as $h^{-1} \nu$ and  $h^{-1} \theta$ , but of course they extend to measures with the same total mass and $W_p^p(\nu, \theta)=h W_p^p(h^{-1} \nu, h^{-1}\theta)$. We shall mainly use the $2$-Wasserstein distance but the $1$-Wasserstein distance will be useful as well in the sequel. We also recall that the $1$-Wasserstein distance can also be defined through the Kantorovich duality formula (see for instance \cite{villani, villani2}):
\begin{equation}\label{kantodual}
W_1(\nu, \theta):=\sup \Big \{ \int_{\R^d} f \mbox{d} (\nu-\theta) \; : \; f \mbox{ $1$-Lipschitz} \Big\}.
\end{equation}

We will see in section \ref{sec-jko} that one may obtain solutions to the system \pref{s1}-\pref{s2}-\pref{s3} by a minimizing scheme for an energy defined on an infinite product of spaces of measures parametrized by the label $a$. Wasserstein gradient flows on finite products have recently been investigated in \cite{dff}, \cite{cl}. To our knowldege the case of an infinite product is new in the literature.

\subsection{Functional setting}

Let $A:=[-R_v, R_v+1]$ and denote by $X$ the set consisting of all $\bn:=(\nu^a)_{a\in A}$, $\mu$-measurable families of measures such that 
\[ \nu^a(\R)=h(a); \; \mbox{ for $\mu$-a.e. $a$ and } \int_A \int_\R x^2  \mbox{d} \nu^a(x) \mbox{d} \mu(a)<+\infty.  \]
Given $R>0$ (the precise choice of $R$ will be made later on, see \pref{choixduR} below), let us denote by $X_R$ the subset of $X$ defined by
\begin{equation}\label{defXR}
X_R:=\{\bn\in X \; : \; \supp(\nu^a)\subset [-R, R], \mbox{ for $\mu$-a.e. $a\in A$}\}.
\end{equation}

For $\bn\in X_R$, let us define the probability (because $\int_\R h(a) \mbox{d} \mu(a)=1$) measure
\[ \bnu:=\int_{\R} \nu^a \mbox{d} \mu(a)\]
and the energy
\begin{equation}\label{defenergy}
J(\bn)=\frac{1}{4} \int_{\R\times \R} \vert x-y\vert  \mbox{d} \bnu(x) \mbox{d} \bnu(y)+\int_A \int_{\R} \Big(\frac{1}{2}-a\Big) x \mbox{d} \nu^a(x) \mbox{d} \mu(a).
\end{equation}
Note that $J$ is unbounded from below on the whole of $X$ but it is bounded on each $X_R$. Note also that the interaction term can be rewritten as:
\begin{equation}\label{expandint}
 \int_{\R\times \R} \vert x-y\vert  \mbox{d} \bnu(x) \mbox{d} \bnu(y)= \int_{\R^4} \vert x-y\vert  \mbox{d} \nu^a(x)\mbox{d} \nu^b(y) \mbox{d} \mu(a) \mbox{d} \mu(b).
\end{equation}
We equip $X_R$ with the distance $d$ given by:
\begin{equation}\label{defded}
d^2(\bn, \bt):=\int_A W_2^2(\nu^a, \theta^a) \mbox{d} \mu(a), \; (\bn, \bt)=((\nu^a)_{a\in A}, (\theta^a)_{a\in A}) \in X_R\times X_R.
\end{equation}
It will also be convenient to work with the weak topology on $X_R$ that is the one defined by the family of semi-norms
\[p_\phi(\bn):=\Big\vert \int_{A\times[-R,R]} \phi \mbox{d}  (\bn\otimes \mu) \ \Big\vert, \; \phi \in C(A\times [-R, R])\]
where $\bn\otimes \mu$ is the probability measure defined by
\[ \int_{A\times [-R,R]}  \phi \mbox{d}  (\bn\otimes \mu)   := \int_{A} \Big(\int_{[-R,R]} \phi(a,x) \mbox{d} \nu^a (x)\Big) \mbox{d} \mu(a)\]
and
\[K:=A\times [-R, R]\]
so that convergence for the weak topology is nothing but weak-$*$ convergence of $\bn\otimes \mu$. Since for all $\bn\in X_R$, $\bn\otimes \mu$ is a probability measure on  the compact set $A\times [-R, R]$, $X_R$ is compact for the weak topology. Note also that since the weak-$*$ topology is metrizable by the Wasserstein distance (see \cite{villani, villani2})  on the set of probability measures on a compact set of $\R^2$, the weak topology is metrizable by the distance $d_w$:
\begin{equation}\label{defdedw}
d_w^2(\bn, \bt):= W_2^2(\bn\otimes \mu, \bt\otimes \mu), \; (\bn, \bt)\in X_R\times X_R,
\end{equation}
so that $(X_R, d_w)$ is a compact metric space. We summarize the basic properties of $J$, $d$ and $d_w$ in the following.

\begin{lem}\label{basic}
Let $X_R$, $J$, $d$ and $d_w$ be defined as above then we have:

\begin{enumerate}

\item $J$ is Lipschitz continuous for $d_w$,

\item $d_w\le d$, 

\item $d$ is lower semicontinous for $d_w$: if $(\bn_n)_n$ is a sequence in $X_R$, $(\bn, \bt)\in X_R\times X_R$ and $\lim_n d_w(\bn_n, \bn)=0$ then $\liminf_n d^2(\bn_n, \bt)\ge d^2(\bn, \bt)$.

\end{enumerate}

\end{lem}

\begin{proof}

Let us recall  that if $\theta$ and $\nu$ are  (compactly supported say) probability measures on $\R^d$ then by Cauchy Schwarz-inequality,
\begin{equation}\label{rappel1}
W_1( \nu, \theta)\le W_2(\nu, \theta)
\end{equation}
and, it follows from \pref{kantodual} that, if $f$ is $M$-Lipschitz then
\begin{equation}\label{rappel2}
\int_{\R^d} f\mbox{d} (\nu-\theta)\le M W_1(\nu, \theta).
\end{equation}
Moreover, 
\begin{equation}\label{rappel3}
W_1(\nu\otimes \nu, \theta\otimes \theta)\le 2 W_1(\nu, \theta).
\end{equation}

1. Let us rewrite $J$ as 
\[J(\bn)=\frac{1}{4} J_0(\bn)+J_1(\bn), \]
with
\begin{equation}\label{J0}
J_0(\bn):=\int_{K^2} \vert x-y\vert \mbox{d} (\bn\otimes \mu)(a,x) \mbox{d} (\bn\otimes \mu)(b,y),
\end{equation}
and
\begin{equation}\label{J1}
J_1(\bn):=\int_K \Big(\frac{1}{2}-a\Big)x  \mbox{d} (\bn\otimes \mu)(a,x).
\end{equation}
The fact that $J_1$ is Lipschitz for $d_w$ directly follows from \pref{rappel1}, \pref{rappel2} and the fact that the integrand in $J_1$ is uniformly Lipschitz in $x$. As for $J_0$,  using also \pref{rappel3} and the fact that the distance is $1$-Lipschitz, we have 
\[\begin{split}
J_0(\bn)-J_0(\bt)\le W_1((\bn\otimes \mu) \otimes (\bn\otimes \mu), (\bt\otimes \mu) \otimes (\bt\otimes \mu))\\
\le 2 W_2(\bn\otimes \mu, \bt\otimes \mu)=2d_w(\bn, \bt).
\end{split}\]

2. Let $\bn=(\nu^a)_{a\in A}$ and $\bt=(\theta^a)_{a\in A}$ be two elements of $X_R$ and let $\gamma^a$ be an optimal plan between $\nu^a$ and $\theta^a$ (which can be chosen in a $\mu$-measurable way, thanks to standard measurable selection arguments, see \cite{cv}). Let us then define the probability measure $\alpha$ on $K^2$ by
\[\begin{split}
&\int_{K\times K} \phi((a,x), (b, y)) \mbox{d} \alpha(a,x, b,y)\\
&:=\int_A \Big( \int_{[-R,R]^2}   \phi((a,x), (a,y)) \mbox{d} \gamma^a(x,y) \Big) \mbox{d} \mu(a)
\end{split}\]
for all $\phi \in C(K\times K)$. Observing that $\alpha \in \Pi(\bn\otimes \mu, \bt\otimes \nu)$, we get
\[\begin{split}
d^2_w(\bn, \bt)\le \int_{K\times K} \vert x-y\vert^2 \mbox{d} \alpha(a,x, b,y) =\int_A \Big( \int_{[-R,R]^2} \vert x-y\vert^2 \mbox{d} \gamma^a (x,y) \Big)\mbox{d} \mu(a)\\
= \int_A   W_2^2 (\nu^a, \theta^a) \mbox{d} \mu(a)=d^2(\bn, \bt).
\end{split}\]


3. Let $\gamma_n^a$ be an optimal plan ($\mu$-measurable with respect to $a$) between $\nu_n^a$ and $\theta^a$. Again passing to a subsequence if necessary we may assume that $\gamma_n^a\otimes \mu$ weakly $*$ converges to some measure of the form $\gamma^a \otimes \mu$. Using test-functions of the form $\psi(a)(\alpha(x)+\beta(y))$ we deduce easily that for $\mu$-almost every $a$, $\gamma^a\in \Pi(\nu^a, \theta^a)$ and then
\[\begin{split}
\liminf_n d^2(\bn_n, \bt) =\liminf \int_{A} \int_{[-R,R]^2} \vert x-y \vert^2 \mbox{d} \gamma_n^a (x,y) \mbox{ d}\mu( a)\\
= \int_A\int_{[-R,R]^2} \vert x-y \vert^2 \mbox{d} \gamma^a (x,y) \mbox{ d}\mu( a) \ge d^2(\bn, \bt).
\end{split}\]

\end{proof}

\subsection{Subdifferential of the energy and gradient flows as measure solutions}

Let us start with some convexity properties of $J$. Let $\bn=(\nu^a)_{a\in A}$ and $\bt$ belong to $X_R$ and let $\bg:=(\gamma^a)_{a\in A}$ be a measurable family of transport plans between $\nu^a$ and $\theta^a$ (which we shall simply denote by $\bg\in \Pi(\bn, \bt)$). For $\eps\in [0,1]$, then define
\begin{equation}\label{interpp}
\bn_\eps:=((1-\eps) \pi_1+\eps\pi_2)_\# \gamma^a)_{a\in A}
\end{equation}
where $\pi_1$ and $\pi_2$ are the canonical projections $\pi_1(x,y)=x$, $\pi_2(x,y)=y$. Then $\eps\in [0,1]\mapsto \bn_\eps$ is a curve which interpolates between $\bn$ and $\bt$. Similarly if we take transport plans $\gamma^a$ induced  by maps of the form $\id +\xi^a$ with $\bxi=(\xi^a)_{a\in A} \in L^{\infty} (\bn \otimes \mu)$  i.e. $\theta^a=(\id+\xi^a)_\#\nu^a$ then $\nu_\eps^a=(\id+\eps \xi^a)_\#\nu^a$ and in this case, we shall simply denote $\bxi:=(\xi^a)_{a\in A}$ and $\bn_\eps$ as 
\[\bn_\eps=(\bid +\eps \bxi)_\#\bn, \; \bt=(\bid + \bxi)_\#\bn.\]

\begin{lem}\label{convex}
Let $\bn$ and $\bt$ be in $X_R$, $\bg\in  \Pi(\bn, \bt)$ and $\bn_\eps$ be given by \pref{interpp}. Then 
\[J(\bn_\eps)\le (1-\eps) J(\bn)+\eps J(\bt), \; \forall \eps\in [0,1].\]
In particular, the same inequality holds  if $\bn_\eps=(\bid +\eps \bxi)_\#\bn$ with $\bxi\in L^{\infty} (\bn \otimes \mu)$.

\end{lem}

\begin{proof}
This immediately follows from the construction of $\bn_\eps$, the convexity of the absolute value in $J_0$ defined by \pref{J0} and the linearity in $x$ of the integrand in $J_1$ defined by \pref{J1}.
\end{proof}

\begin{defi}\label{defsousdiff}
Let $\bn\in X_R$, the subdifferential of $J$ at $\bn$, denoted $\partial J(\bn)$,  consists of all $\bw:=(w^a)_{a\in A} \in L^1(\bn \otimes \mu)$ such that for every $R'>0$, every $\bt\in X_{R'}$ and every $\bg=(\gamma^a)_{a\in A} \in \Pi(\bn, \bt)$, one has
\[J(\bt)-J(\bn) \ge \int_ {[-R,R]\times [-R', R']\times A}   w^a (y) (z-y)    \mbox{d} \gamma^a(y,z) \mbox{d}\mu(a).\]
\end{defi}

\begin{rem}\label{sousdiffdes}
An equivalent way  to define $\partial J(\bn)$  (which will turn out to be more convenient in the sequel to prove stability properties, see Lemma \ref{sdok}) is in terms of transition kernels rather than of transport plans. More precisely, given $\bn\in X_R$, we define the set $T(\bn)$ of $\bn\otimes \mu$ measurable  maps  $\bet$: $(a,y)\in K\mapsto \eta^{a,y}\in \PP(\R)$ such that there exists an $R'>0$ such that $\eta^{a,y}$ is supported by $[-R', R']$ for  $\bn\otimes \mu$ almost every $(a,y)\in K$. We then define $\bn_{\bet}=(\nu^a_{\bet})_{a\in A}$ by 
\[\int_{\R} \varphi(z) \mbox{d} \nu^a_{\bet}(z):=\int_{\R^2} \varphi(z) \mbox{d} \eta^{a,y} (z) \mbox{d} \nu^a (y), \; \forall \varphi \in C(\R).\] 
By construction, $\bg=(\gamma^a)_{a\in A}$ with $\gamma^a =\nu^a \otimes \eta^{a,y}$ defined by
\[\int_{\R^2} \varphi(y,z) \mbox{d} \gamma^a(y,z):=\int_{\R^2} \varphi(y,z) \mbox{d} \eta^{a,y} (z) \mbox{d} \nu^a (y), \; \forall \varphi \in C(\R^2)\] 
 belongs to $\Pi(\bn, \bn_{\bet})$ and thanks to the disintegration Theorem, it is then easy to check that  $\bw\in \partial J(\bn)$ if and only if, for every $\bet\in T(\bn)$, one has
\begin{equation}\label{sousdiffkernel}
J(\bn_\bet)-J(\bn)\ge \int_{\R^3} w^a (y)(z-y) \mbox{d} \eta^{a,y} (z)  \mbox{d} \nu^a (y) \mbox{d} \mu(a). 
\end{equation}

\end{rem}

\begin{rem}
If we restrict ourselves to transport maps (i.e. take $\eta^{a,y}=\delta_{\xi^a(y)}$ in \pref{sousdiffkernel}), we obtain a  condition which is weaker than definition \ref{defsousdiff} but somehow easier to handle. If $\bw:=(w^a)_{a\in A} \in L^1(\bn \otimes \mu)\in \partial J(\bn)$ then for every $\bxi=(\xi^a)_{a\in A} \in L^{\infty} (\bn \otimes \mu)$, one has 
\begin{equation}\label{sousdiffmap}
J((\bid + \bxi)_\#\bn)-J(\bn) \ge  \int_K \bw \bxi   \mbox{d}(\bn\otimes \mu)=\int_K w^a (x) \xi^a (x)  \mbox{d} \nu^a(x) \mbox{d}\mu(a).
\end{equation}
\end{rem}

\begin{rem}
The subdifferential $\partial J$ obviously has the following monotonicity property (which will be crucial for uniqueness, see section \ref{sec-wp}) : if $\bn_1$ and $\bn_2$ belong to $X_R$ and $\bw_1\in \partial J(\bn_1)$ and $\bw_2\in \partial J(\bn_2)$, then for every $\bg\in \Pi(\bn_1, \bn_2)$, one has
\begin{equation}\label{monsousdiff}
\int_{\R^3}  (w_1^a(y)-w_2^a(z))(y-z) \mbox{d} \gamma^a(y,z) \mbox{d} \mu(a)\ge 0. 
\end{equation}

\end{rem}


The connection between the subdifferential (in fact the weak condition \pref{sousdiffmap})  of the energy $J$ given by \pref{defenergy} and the condition \pref{s2} is clarified by the following:

\begin{prop}\label{linksys}
Let $\bn\in X_R$, if $\bw\in \partial J(\bn)$ then, defining the $a$-marginal of $\bn \otimes \mu$ by
\[\rho:=\int_A \nu^a \mbox{d} \mu(a)\]
and its cumulative distribution function by
\[G(x):=\rho((-\infty, x]),\; G^{-}(x)=\rho((-\infty, x)), \; \forall x\in \R,\]
we have
\begin{equation}\label{localsd}
w^a(x)\in [G^{-}(x)-a, G(x)-a]  \mbox{ for $\bn \otimes \mu$ a.e. $(a,x)$}.
\end{equation}
In particular $\bw \in L^{\infty} (\bn \otimes \mu)$ with 
\begin{equation}\label{boundsdiff}
\Vert \bw \Vert_{L^{\infty} (\bn \otimes \mu)} \le R_v+2. 
\end{equation}
\end{prop}

\begin{proof}
Let $\bxi \in L^{\infty} (\bn \otimes \mu)$ and define $\bn_\eps:=(\bid+\eps \bxi)_\# \bn$ for $\eps\in[0,1]$. Since $\bw\in \partial J(\bn)$ we have in particular
\begin{equation}\label{sd1}
\lim_{\eps \to 0^+} \frac{1}{\eps} (J(\bn_\eps)-J(\bn)) \ge \int_K \bw \bxi   \mbox{d}(\bn\otimes \mu)= \int_K w^a (x) \xi^a(x) \mbox{d} \nu^a(x) \mbox{d} \mu(a). 
\end{equation}
Defining $J_0$ and $J_1$ as in \pref{J0}-\pref{J1} and $K:=A\times [-R, R]$ , first we have
\begin{equation}\label{sd2}
\frac{1}{\eps} (J_1(\bn_\eps)-J_1(\bn)) =I_0:= \int_K \Big(\frac{1}{2}-a\Big) \xi^a(x) \mbox{d} \nu^a(x) \mbox{d} \mu(a). 
\end{equation}
We then write
\begin{equation}\label{sd3}
\frac{1}{\eps} (J_0(\bn_\eps)-J_0(\bn)) = \int_{K\times K} \eta_\eps(a,b,x, y) \mbox{d}( \bn\otimes \mu)(a,x) \mbox{d}( \bn\otimes \mu)(b,y)
\end{equation}
with
\begin{equation}\label{sd4}
 \eta_\eps(a,b,x, y) =\frac{1}{\eps} \Big( \vert x +\eps \xi^a (x)-(y+\eps \xi^b(y))\vert - \vert x-y\vert \Big). 
\end{equation}
Observing that $\eta_\eps$ is bounded by $2 \Vert \bxi\Vert_{L^{\infty} (\bn \otimes \mu)}$ and that 
\begin{equation}\label{sd5}
\lim_{\eps \to 0^+} \eta_\eps(a,b,x, y)=\begin{cases} \sign(x-y) (\xi^a(x)-\xi^b(y)), \mbox{ if $x\neq y$} \\  \vert \xi^a(x)-\xi^b(y)\vert , \mbox{ if $x=y$,}   \end{cases}
\end{equation}
by Lebesgue's dominated convergence theorem, we get 
\begin{equation}\label{sd6}
\lim_{\eps \to 0^+} \frac{1}{\eps} (J(\bn_\eps)-J(\bn))=I_0+I_1+I_2
\end{equation}
with $I_0$ given by \pref{sd2}, and 
\begin{equation}\label{sd7}
I_1=\frac{1}{4} \int_{K\times K}   \indi_{x\neq y}\sign(x-y)  (\xi^a(x)-\xi^b(y))  \mbox{d}( \bn\otimes \mu)(a,x) \mbox{d}( \bn\otimes \mu)(b,y) 
\end{equation}
and 
\begin{equation}\label{sd8}
I_2=\frac{1}{4} \int_{K\times K}  \indi_{x=y}  \vert \xi^a(x)-\xi^b(x)\vert  \mbox{d}( \bn\otimes \mu)(a,x) \mbox{d}( \bn\otimes \mu)(b,y). 
\end{equation}
To compute $I_1$ we observe that thanks to Fubini's theorem
\[\begin{split}
&\frac{1}{4} \int_{K\times K}  \indi_{x>y}  (\xi^a(x)-\xi^b(y))  \mbox{d}( \bn\otimes \mu)(a,x) \mbox{d}( \bn\otimes \mu)(b,y) \\
&= \frac{1}{4} \int_{K}   \xi^a(x) G^{-}(x) \mbox{d}( \bn\otimes \mu)(a,x) -\frac{1}{4} \int_{K}   \xi^b(y) (1-G(y)) \mbox{d}( \bn\otimes \mu)(b,y)\\
&= \frac{1}{4} \int_{K}   \xi^a(x) (G^-(x)+G(x)-1) \mbox{d}( \bn\otimes \mu)(a,x).
\end{split}\]
Treating similarly the integral on $\{x<y\}$ we thus get
\begin{equation}\label{sd9}
I_1= \int_{K}   \Big(\frac{G^-(x)+G(x)}{2}-\frac{1}{2}\Big)  \xi^a(x) \mbox{d}( \bn\otimes \mu)(a,x).
\end{equation}
As for $I_2$, we have
\begin{equation}\label{sd10}
I_2\le \frac{1}{4} \int_{A\times A} \Big( \int_{[-R, R]}  \Big( \vert \xi^a(x)\vert +\vert \xi^b(x)\vert \Big) \nu^b(\{x\}) \mbox{d} \nu^a(x)  \Big) \mbox{d} \mu(a)  \mbox{d} \mu(b),
\end{equation}
then we use Fubini's theorem to get
\[\begin{split} \int_{A\times A} \Big( \int_{[-R, R]}  \vert \xi^a(x)\vert  \nu^b(\{x\}) \mbox{d} \nu^a(x)  \Big) \mbox{d} \mu(a)  \mbox{d} \mu(b)\\
=\int_K \vert \xi^a (x)\vert (G(x)-G^-(x))   \mbox{d}( \bn\otimes \mu)(a,x).
\end{split}\]
Note that in the previous integral, the integration with respect to $x$ is actually a discrete sum, because the set of atoms where $G>G^{-}$ is at most countable since $G$ is nondecreasing; let us denote this set by
\[S:=\{x\in [-R, R] \; : \; G(x)-G^{-}(x)>0\}=\{x_i\}_{i\in I}\]
where $I$ is at most countable. Similarly  for the second term in the right hand side of \pref{sd10} 
observing that $\vert \xi^b(x)\vert  \int_A  \nu^b(\{x\})   \mbox{d} \mu(b)  \le  \Vert \bxi\Vert_{L^{\infty} (\bn \otimes \mu)} (G(x)-G^{-}(x))$, we only have to integrate in $x$ over $S$ which gives
\[\begin{split} 
&\int_{A\times A} \Big( \int_{[-R, R]}  \vert \xi^b(x)\vert  \nu^b(\{x\}) \mbox{d} \nu^a(x)  \Big) \mbox{d} \mu(a)  \mbox{d} \mu(b)\\
&=\int_{A\times A} \Big( \sum_{i\in I}  \vert \xi^b(x_i)\vert  \nu^b(\{x_i\}) \nu^a(\{x_i\})  \Big) \mbox{d} \mu(a)  \mbox{d} \mu(b)\\
&=\int_{A} \Big( \sum_{i\in I}  \vert \xi^b(x_i)\vert  \nu^b(\{x_i\}) (G(x_i)-G^{-}(x_i)  \Big)  \mbox{d} \mu(b)\\
&=\int_K \vert \xi^b (x)\vert (G(x)-G^-(x))   \mbox{d}( \bn\otimes \mu)(b,x),
\end{split}\]
so that 
\begin{equation}\label{sd11}
I_2\le \frac{1}{2} \int_K \vert \xi^a (x)\vert (G(x)-G^-(x))   \mbox{d}( \bn\otimes \mu)(a,x). 
\end{equation}
Putting together \pref{sd1}, \pref{sd2}, \pref{sd6}, \pref{sd9} and \pref{sd11} we arrive at the inequality
\[\begin{split} 
  \int_K \Big(w^a (x)+a-\frac{1}{2}(G(x)+G^-(x))\Big)   \xi^a(x)   \mbox{d} (\bn\otimes \mu)(a,x)\\ \le  \frac{1}{2} \int_K \vert \xi^a (x)\vert (G(x)-G^-(x))   \mbox{d}( \bn\otimes \mu)(a,x)   
\end{split}\]
which holds for any $\bxi \in L^{\infty} (\bn \otimes \mu)$ and \pref{localsd} obviously follows.

\end{proof}

\begin{defi}\label{defgf}
A gradient flow of $J$ on the time interval $[0,T]$ starting from $\bn_0$ is a Lipschitz continuous (for $d$) curve $t\in [0,T]\mapsto \bn(t)=(\nu(t)^a)_{a\in A} \in X_R$ together with a measurable map $t\in [0,T]\mapsto \bv(t) \in  L^1(\bn \otimes \mu)$ such that $\bv(t) \in -\partial J(\bn(t))$ for almost every $t\in[0,T]$, and for $\mu$-almost every $a\in A$, $t\mapsto \nu(t)^a$ is a solution in the sense of distributions of the continuity equation \pref{s1}.
\end{defi}

It then follows from Proposition \ref{linksys} that gradient flows starting from $\bn_0$ are measure solutions of the system \pref{s1}-\pref{s2}-\pref{s3}. Note also that thanks to the bound \pref{boundsdiff}, gradient flows are not only absolutely continuous but automatically  Lipschitz for $d$ and even more is true: for $\mu$-almost every $a$, the curve $t\mapsto \nu_t^a$ is Lipschitz for $W_2$, more precisely
\begin{equation}\label{lipent}
W_2(\nu_t^a, \nu_s^a)\le \vert t-s\vert (R_v+2)  h(a)^{1/2} \mbox{ hence } d(\bn(t), \bn(s))\le \vert t-s\vert (R_v+2).
\end{equation}

\section{Existence by the JKO scheme}\label{sec-jko}

We will prove existence of a gradient flow curve on the time interval $[0,T]$ starting from $\bn_0=(\nu_0^a)_{a\in A}$ by considering the JKO scheme. Given a time step $\tau>0$, starting from $\bn_0$, we construct inductively a sequence $\bn_k$ by 
\begin{equation}\label{jkoscheme}
\bn_{k+1}\in \argmin_{\bn \in X} \Big\{\frac{1}{2\tau} d^2(\bn, \bn_k)+J(\bn)\Big\}
\end{equation}
for $k=0, \cdots, N$ with $N:=[\frac{T}{\tau}]$.

\subsection{Estimates}

The first step in proving that this scheme is well-defined consists in showing that one can a priori bound the support. This is based on the following basic observation:

\begin{lem}\label{boundsupport}
Let $R_0$, $R>0$ and $\tau$ be positive constants, $\nu_0$ be a probability measure on $\R^d$ with support in $B_{R_0}$ and $\nu \in \PP_2(\R^d)$. Let $P$ be the projection onto $B_{R_0+\tau R}$ and define $\hnu:=P_\# \nu$. Then, for every $a\in B_R$, one has
\[\frac{1}{2} W_2^2(\hnu, \nu_0)-\tau \int_{\R^d}  a \cdot x \mbox{d} \hnu(x) \le \frac{1}{2} W_2^2(\nu, \nu_0)-\tau \int_{\R^d}  a \cdot x \mbox{d} \nu(x).\]

\end{lem}

\begin{proof}
Fix an optimal transport plan between $\nu_0$ and $\nu$ i.e. a $\gamma\in \Pi(\nu_0, \nu)$ such that $W_2^2(\nu, \nu_0)=\int_{\R^d\times \R^d} \vert x-y \vert^2 \mbox{d} \gamma(x,y)$. Since the map $(x,y)\mapsto (x, P(y))$ pushes forward $\gamma$ to a plan having $\nu_0$ and $\hnu$ as marginals, we have
\[\begin{split}
\frac{1}{2} W_2^2(\hnu, \nu_0) \le \frac{1}{2}  \int_{\R^d\times \R^d} \vert x-P(y)\vert^2 \mbox{d} \gamma(x,y)= \frac{1}{2} W_2^2(\nu, \nu_0) \\
-\frac{1}{2} \int_{\R^d\times \R^d} \vert y-P(y)\vert^2 \mbox{d} \gamma(x,y) +  \int_{\R^d\times \R^d} (y-P(y)) \cdot(x-P(y)) \mbox{d} \gamma(x,y)
\end{split}\]
and then
\[\begin{split}\frac{1}{2} W_2^2(\hnu, \nu_0)-\tau \int_{\R^d}  a \cdot x \mbox{d} \hnu(x) -\frac{1}{2} W_2^2(\nu, \nu_0)+\tau \int_{\R^d}  a \cdot x \mbox{d} \nu(x)\\
\le \int_{\R^d\times \R^d} (y-P(y)) \cdot(x+\tau a-P(y)) \mbox{d} \gamma(x,y).
\end{split}\]
But since $\gamma$-a.e. $x+\tau a \in B_{R_0+\tau R}$, we get that the integrand in the right-hand side is nonpositive  by the well-known characterization of the projection onto $B_{R_0+\tau R}$. 
\end{proof}  

Now consider the first step of the JKO scheme. Since $\nu_0^a$ is supported by $[-R_x, R_x]$, for every $a\in A$ and $a\in A\Rightarrow \vert a \vert \leq R_v +1$, the previous lemma implies that if one replaces $\bn=(\nu^a)_{a\in A}\in X$ by $\hat{\bn}=(\hat{\nu}^a)_{a\in A}$ defined for every $a$ by $\hat{\nu}^a=P_\#\nu^a$ where $P$ is the projection on $[-R_x-\tau (R_v+3/2), R_x+\tau (R_v+3/2)]$, one has
\[\frac{1}{2} W_2^2(\hnu^a, \nu^a_0)+\tau \int_{\R^d} \left(\frac{1}{2}-a\right) \cdot x \mbox{d} \hnu^a(x) \le \frac{1}{2} W_2^2(\nu^a, \nu^a_0)+\tau \int_{\R^d}  \left(\frac{1}{2}-a\right) \cdot x \mbox{d} \nu^a(x).\]
As for the interaction term, it is also improved by replacing $\bn$ by  $\hat{\bn}$; this is obvious from the expression \pref{expandint} and the fact that $P$ is $1$-Lipschitz. In the first step of the JKO scheme, we may therefore impose the constraint that $\bn\in X_{R_x+\tau (R_v+3/2)}$. After $k$ steps, we may similarly impose that the minimization is performed on $X_{R_x+k \tau (R_v+3/2)}$, so simply setting
\begin{equation}\label{choixduR}
R=R_x+(T+\tau) (R_v+3/2),
\end{equation}
we may replace \pref{jkoscheme} with a bound on the support:
\begin{equation}\label{jkoschemeR}
\bn_{k+1}\in \argmin_{\bn \in X_R} \Big\{\frac{1}{2\tau} d^2(\bn, \bn_k)+J(\bn)\Big\}.
\end{equation}

By a direct application of Lemma \ref{basic} and the compactness of $(X_R, d_w)$, we then see that  the minimizing scheme \pref{jkoschemeR} is well-defined and actually defines a sequence $\bn_k$, $k=0,\cdots, N+1$. We also extend this sequence by piecewise constant in time interpolation:
\begin{equation}\label{interpol}
\bn_{\tau}(t):=\bn_k, \mbox{ for $t\in((k-1)\tau, k\tau]$}, \; k=1, \cdots, N+1. 
\end{equation}

In the following basic estimates, $C$ will denote a constant (possibly depending on $T$) which may vary from one line to the other.  By construction, for all $k=0,\ldots, N$, we have
\begin{equation}\label{estim1}
\frac{1}{2\tau} d^2(\bn_{k+1}, \bn_k)\le J(\bn_k)-J(\bn_{k+1}).
\end{equation}
Summing and using the fact that every $\bn_k$ belongs to $X_R$ and that $J$ is bounded from below on $X_R$ we get:
\begin{equation}\label{estim2}
\frac{1}{2\tau} \sum_{k=0}^N d^2(\bn_{k+1}, \bn_k)\le J(\bn_0)-J(\bn_{N+1})\le C.
\end{equation}
From \pref{estim2}, Cauchy-Schwarz inequality and Lemma \ref{basic} we classically get a uniform H\"older estimate:
\begin{equation}\label{estim3}
d_w(\bn_\tau(t), \bn_\tau(s))\le d(\bn_\tau(t), \bn_\tau(s))\le C \sqrt{ \vert t-s\vert +\tau}, \; \forall (s,t)\in [0, T]^2.  
\end{equation}
Since $(X_R, d_w)$ is a compact metric space, it follows from some refined variant of Ascoli-Arzel\`a theorem (see \cite{AGS}) that there exists a limit curve 
\[t\mapsto \bn(t) \mbox{ belonging to $C^{0,\frac{1}{2}}([0,T], (X_R, d_w))$}\]
and a vanishing sequence of time-steps $\tau_n \to 0$ as $n\to+\infty$ such that 
\begin{equation}\label{estim4}
\sup_{t\in [0,T]} d_w(\bn_{\tau_n}(t), \bn(t)) \to 0 \mbox{ as $n\to +\infty$}.   
\end{equation}

\subsection{Discrete Euler-Lagrange equation}

Let $\bg_{k+1}=(\gamma_{k+1}^a)_{a\in A} \in \Pi(\bn_{k}, \bn_{k+1})$ be such that $\gamma^a_{k+1}$ is an optimal plan for $\mu$-almost every $a$ and let $v_{k+1}^a$ be defined by
\[\int_{[-R,R]} \xi(y) v_{k+1}^a(y) \mbox{d} \nu_{k+1}^a (y)= \int_{[-R,R]^2} \xi(y)\frac{y-x}{\tau} \mbox{d} \gamma_{k+1}^a(x,y)\]
for all $\xi \in C([-R, R])$, or equivalently, disintegrating $\gamma_{k+1}^a$ with respect to its second marginal $\nu_{k+1}^a$ as $\mbox{d} \gamma_{k+1}^a(x,y)= \mbox{d} \gamma_{k+1}^{a,y}(x)  \otimes \mbox{d} \nu_{k+1}^a(y)$: 
\begin{equation}\label{defvk}
 v_{k+1}^a(y)=\frac{1}{\tau}\Big( y-\int_{[-R,R]}x  \mbox{d} \gamma_{k+1}^{a,y}(x)\Big).
\end{equation}

The Euler-Lagrange equation for \pref{jkoscheme} can then be written as
\begin{lem}\label{lem-el}
Let $\bn_{k+1}$ be a solution of \pref{jkoscheme},  $\bg_{k+1} \in \Pi(\bn_{k}, \bn_{k+1})$ and $\bv_{k+1}$ be constructed as above, then:
\begin{equation}\label{el}
\bv_{k+1}\in -\partial J (\bn_{k+1}).
\end{equation}
\end{lem}

\begin{proof}
Let $R'>0$, $\bt\in X_{R'}$ and $\bg\in \Pi(\bn^{k+1}, \bt)$, and define for $\eps\in [0,1]$ 
\[\bn_\eps= (\nu_\eps^a)_{a\in A} \mbox{ with } \nu_\eps^a :=((1-\eps) \pi_1+\eps\pi_2)_\#\gamma^a.\]
Then by optimality of $\bn_{k+1}$ and using Lemma \ref{convex}, we have
\[\begin{split}
0\le \liminf_{\eps \to 0^+} \frac{1}{\eps} \Big(  \frac{1}{2\tau} (d^2(\bn_\eps, \bn_k)-d^2(\bn_{k+1}, \bn_k)) + J(\bn_\eps)-J(\bn_{k+1})  \Big)\\
\le \liminf_{\eps \to 0^+} \frac{1}{\eps} \Big(  \frac{1}{2\tau} (d^2(\bn_\eps, \bn_k)-d^2(\bn_{k+1}, \bn_k) \Big) +J(\bt)-J(\bn_{k+1}).
\end{split}\]
We have already disintegrated the optimal plan $\gamma_{k+1}^a $ between $\nu_k^a$ and $\nu_{k+1}^a$ as
\[     \gamma_{k+1}^a (\mbox{d}x,\mbox{d} y)= \gamma_{k+1}^{a,y} ( \mbox{d}x) \otimes \nu_{k+1}^a  (\mbox{d}y). \]
Let us also disintegrate the (arbitrary) plan $\gamma^a $ between $\nu_{k+1}^a$ and $\theta^a$ as:
\[\gamma^a (\mbox{d}y,\mbox{d} z)= \nu_{k+1}^a  (\mbox{d}y)  \otimes \gamma^{a,y} (\mbox{d}z).\] 
Define then the $3$-plan $\beta^a$ by $\beta^a =( \gamma_{k+1}^{a,y} \otimes \gamma^{a,y}  )\otimes \nu_{k+1}^a$ i.e. 
\[\int_{\R^3} \phi(x,y,z) \mbox{d} \beta^a (x,y,z):=\int_{\R} \Big( \int_{\R^2} \phi(x,y,z)  \mbox{d} \gamma_{k+1}^{a,y}(x) \mbox{d} \gamma^{a,y}(z)   \Big) \mbox{d} \nu_{k+1}^a (y) \]
for every $\phi \in C(\R^3)$. 
Setting 
\[\begin{split}
(\pi_1(x,y,z), \pi_2(x,y,z), \pi_3(x,y,z))=(x,y,z), \\
(\pi_{12}(x,y,z), \pi_{23} (x,y,z), \pi_{13}(x,y,z))=( (x,y), (y,z), (x,z)),
\end{split}\]
we have by construction, ${\pi_{12}}_\# \beta^a =\gamma_{k+1}^a$, ${\pi_{23}}_\#\beta^a=\gamma^a$. By the very definition of $\nu_\eps^a$, we also have $(\pi_1, (1-\eps)\pi_2+\eps \pi_3))_\#\beta^a \in \Pi(\nu_k^a, \nu_\eps^a)$ so that
\[W_2^2(\nu_k^a, \nu_{k+1}^a)= \int_{\R^3} \vert y-x\vert^2 \mbox{d} \beta^a(x,y,z)\]
and
\[W_2(\nu_k^a, \nu_\eps^a)\le \int_{\R^3} \vert (1-\eps)y+\eps z-x\vert^2 \mbox{d} \beta^a(x,y,z).\]
Using Lebesgue's dominated convergence Theorem and recalling the definition of $\beta^a$ and $v_{k+1}^a$ we then get
\[\begin{split}
 \liminf_{\eps \to 0^+} \frac{1}{\eps} \Big(  \frac{1}{2\tau} (d^2(\bn_\eps, \bn_k)-d^2(\bn_{k+1}, \bn_k) \Big)\le\\
 \int_A \Big(\int_{\R^3}  (z-y)\cdot \frac{y-x}{\tau} \mbox{d} \beta^a (x,y,z) \Big) \mbox{d} \mu(a)\\
=  \int_A \Big (\int_{\R^2}     \Big(    \int_{[-R,R]}  \frac{y-x}{\tau} \mbox{d} \gamma_{k+1}^{a,y}(x) \Big)   (z-y)   \mbox{d} \gamma^{a,y}(z) \mbox{d} \nu_{k+1}^a(y)     \Big) \mbox{d} \mu(a)\\
=\int_ {[-R,R]\times [-R', R']\times A}   v^a_{k+1} (y) \cdot (z-y)    \mbox{d} \gamma^a(y,z) \mbox{d}\mu(a).
 \end{split}\]
This yields
\[J(\bt)-J(\bn_{k+1})\ge -\int_ {[-R,R]\times [-R', R']\times A}   v^a_{k+1} (y) \cdot (z-y)    \mbox{d} \gamma^a(y,z) \mbox{d}\mu(a)\]
i.e. $\bv_{k+1}\in -\partial J (\bn_{k+1})$. 
\end{proof}

Let us also extend $v_{k+1}$ by piecewise constant interpolation
\begin{equation}
\bv_{\tau}(t)=\bv_{k+1},  \; t\in ((k\tau, (k+1)\tau], \; t\in [0, T], \; \bv_{k+1}=(v^a_{k+1})_{a\in A},
\end{equation}
so that, thanks to the previous Lemma, we have
\begin{equation}\label{sousgradtau}
\bv_\tau(t)\in -\partial J(\bn_\tau(t)), \; t\in [0, T].
\end{equation}

Thanks to Proposition \ref{linksys}, note that $\sup_{t\in [0,T]}  \Vert \bv_{\tau}(t)\Vert_{L^{\infty}(\bn_\tau(t)\otimes \mu)} \le C$; we can then define the time-dependent-family of signed measures
\[ \mbox{d}\bq_\tau(t)= \bv_\tau(t) \mbox{d} \bn_{\tau}(t), \mbox{ i.e. }  \mbox{d} q_\tau(t)^a= v_\tau(t)^a \mbox{d} \nu_{\tau}(t)^a.\]
Denoting by $\lambda$ the one dimensional Lebesgue measure on $[0,T]$, we may assume, taking a subsequence if necessary, that the bounded family of measures on $\bq_{\tau_n} \otimes \mu \otimes \lambda$ converges weakly $*$ to some bounded signed measure on $[-R,R]\times A\times [0,T]$ which is necessarily  of the form $\bq \otimes \mu \otimes \lambda$ because marginals (with respect to the $a$ and $t$ variables) are stable under weak limits. Since $\vert \bq_{\tau_n}\vert \otimes \mu \otimes \lambda  \le C \bn_{\tau_n}   \otimes \mu \otimes \lambda$ and $\bn_{\tau_n}\otimes \mu$ converges weakly $*$ to $\bn\otimes \mu$,  we have $\vert \bq \vert  \otimes \mu \otimes \lambda  \le C \bn   \otimes \mu \otimes \lambda$.   Hence, for $\mu\otimes \lambda$ a.e. $(a,t)$, the limit satisfies  $\vert q(t)^a \vert \le C \nu(t)^a$ and therefore can be written in the form $\mbox{d} q(t)^a=v(t)^a \mbox{d} \nu^a(t)$  ($\bq= \bv \bn$ for short) with $\Vert \bv(t)\Vert_{L^{\infty}(\bn(t)\otimes \mu)} \le C$ for $\lambda$-a.e. $t\in [0,T]$. We thus have
\begin{equation}\label{cvq}
\bq_{\tau_n} \otimes \mu \otimes \lambda=  (\bv_{\tau_n} \bn_{\tau_n}) \otimes \mu \otimes \lambda \ws \bq   \otimes \mu \otimes \lambda = (\bv \bn) \otimes \mu \otimes \lambda \mbox{ as $n\to +\infty$}.
\end{equation}
In other words,  for every $\phi\in C([0,T]\times A\times[-R, R])$ we have
\[\begin{split}
\lim_n\int_0^T \int_{A} \Big(\int_{[-R, R]} \phi(t,a,x) v_{\tau_n}(t)^a(x) \mbox{d} \nu_{\tau_n}(t)^a(x) \Big) \mbox{d} \mu(a) \mbox{d}t\\
=\int_0^T \int_{A} \Big(\int_{[-R, R]} \phi(t,a,x) v(t)^a(x) \mbox{d} \nu(t)^a(x) \Big) \mbox{d} \mu(a) \mbox{d}t.
\end{split}\]

\subsection{Existence by passing to the limit}

Our task now consists in showing that the limit curve $t\mapsto \bn(t)$ is a gradient flow solution associated to the velocity $t\mapsto \bv(t)$ constructed above.  Let us first check that it satisfies the system of continuity equations \pref{s1}. To do so, take test functions $\psi\in C(A)$ and $\phi \in C^2([0,T]\times [-R,R])$ and let us consider
\[\begin{split}
&\int_0^{N\tau}  \Big( \int_{K} \psi(a) \partial_t \phi(t,x) \mbox{d} \nu_{\tau}(t)^a(x) \mbox{d} \mu(a)    \Big) \mbox{d} t\\
&=\int_A \psi(a) \Big(  \sum_{k=0}^{N-1}  \int_{-R}^R (\phi((k+1)\tau, x)-\phi(k\tau,x)) \mbox{d} \nu_{k+1}^a(x) \Big)  \mbox{d} \mu(a).
\end{split}\]
Then, we rewrite 
\[\begin{split}
\sum_{k=0}^N  \int_{-R}^R (\phi((k+1)\tau, x)-\phi(k\tau,x)) \mbox{d} \nu_{k+1}^a(x)\\
=\sum_{k=1}^{N-1}  \int_{-R}^R \phi(k\tau,x)) \mbox{d} (\nu_k^a- \nu_{k+1}^a)(x)\\
+   \int_{-R}^R \phi(N\tau, x)  \mbox{d} \nu_{N}^a(x)- \int_{-R}^R \phi(0, x)  \mbox{d} \nu_1^a(x).
\end{split}\]
Using the optimal plans $\gamma^a_{k+1}$ as in Lemma \ref{lem-el}, we then rewrite
\[\int_{-R}^R \phi(k\tau,x)) \mbox{d} (\nu_k^a- \nu_{k+1}^a)(x)=\int_{-R}^R  \int_{-R}^R (\phi(k\tau,x)-\phi(k\tau,y)) \mbox{d} \gamma_{k+1}^a(x,y).\]
A Taylor expansion gives
\[ \phi(k\tau,x)-\phi(k\tau,y)=\partial_x \phi(k\tau, y)(x-y)+ l_k(\tau, a, x,y), \; \vert l_k(\tau, a, x,y)\vert \leq \Vert \partial_{xx} \phi \Vert_{\infty} \vert x-y\vert^2.\]
Integrating and using the optimality of $\gamma^a_{k+1}$ gives
\[ l_k(\tau, a):= \int_{-R}^R  \int_{-R}^R \vert l_k(\tau, a, x,y)\vert   \mbox{d} \gamma_{k+1}^a(x,y)  \leq   \Vert \partial_{xx} \phi \Vert_{\infty} W_2^2(\nu_k^a, \nu_{k+1}^a)\]
and then, recalling \pref{estim2} we have
\begin{equation}\label{err1}
\int_A   \psi(a) \sum_{k=1}^{N-1} l_k(\tau, a)  \mbox{d} \mu(a)  \le C \tau \Vert \partial_{xx} \phi \Vert_{\infty} \Vert \psi \Vert_{\infty}.
\end{equation}

Recalling the definition of the discrete velocity $v_{k+1}$ from Lemma \ref{lem-el}, we can rewrite
\[\int_{-R}^R  \int_{-R}^R \partial_x  \phi(k\tau, y)(x-y)  \mbox{d} \gamma_{k+1}^a(x,y)=-\tau \int_{-R}^R \partial_x \phi(k\tau, x) v_{k+1}^a(x) \mbox{d} \nu^a_{k+1}(x),\]
hence by definition of $\bn_\tau$ and $\bv_\tau$ 
\[\begin{split}
\int_A \psi(a) \Big(\sum_{k=1}^{N-1} \int_{-R}^R  \int_{-R}^R \partial_x  \phi(k\tau, y)(x-y)  \mbox{d} \gamma_{k+1}^a(x,y)\Big) \mbox{d} \mu(a)\\
=- \int_0^T \int_K \psi(a) \partial_x \phi(t,x) v_\tau(t)^a \mbox{d}\nu_{\tau}(t)^a(x) \mbox{d} \mu(a) \mbox{d}t +O(\tau).
\end{split}\]
Now thanks to \pref{estim4}, we have 
\begin{equation}\label{err2}
\lim_n  \int_A \psi(a) \Big( \int_{-R}^R \phi(N\tau_n, x)  \mbox{d} \nu_{N}^a(x) \Big)  \mbox{d} \mu(a)=\int_A \psi(a) \Big( \int_{-R}^R \phi(T, x)  \mbox{d} \nu(T)^a(x) \Big)  \mbox{d} \mu(a)
\end{equation}
and
\begin{equation}\label{err3}
\lim_n  \int_A \psi(a) \Big (\int_{-R}^R \phi(0, x)  \mbox{d} \nu_1^a(x))  \Big)  \mbox{d} \mu(a)=\int_A \psi(a) \Big( \int_{-R}^R \phi(0, x)  \mbox{d} \nu_{0}^a(x) \Big)  \mbox{d} \mu(a),
\end{equation}
where we use in the above limits that $\nu_{N}^a=\nu^a_{\tau_n}(N\tau_n)$ and $\nu^a_1=\nu^a_{\tau_n}(\tau_n)$.
Putting the previous computations together, summing  and using \pref{err2}, \pref{err1}, \pref{err3}, we thus obtain
\[\begin{split}
\int_0^{N\tau}  \Big( \int_{K} \psi(a) \partial_t \phi(t,x) \mbox{d} \nu_{\tau}(t)^a(x) \mbox{d} \mu(a)    \Big) \mbox{d} t\\
=- \int_0^T \int_K \psi(a) \partial_x \phi(t,x) v_\tau(t)^a \mbox{d}\nu_{\tau}(t)^a(x) \mbox{d} \mu(a) \mbox{d}t  \\
+ \int_A \psi(a) \Big( \int_{-R}^R \phi(T, x)  \mbox{d} \nu(T)^a(x) \Big)  \mbox{d} \mu(a)\\
-\int_A \psi(a) \Big( \int_{-R}^R \phi(0, x)  \mbox{d} \nu_{0}^a(x) \Big)  \mbox{d} \mu(a)+\eps_\tau
\end{split}\]
where $\eps_{\tau_n}$ goes to $0$ as $n\to +\infty$. Taking $\tau=\tau_n$, using  \pref{estim4}, \pref{cvq} and letting  $n\to +\infty$ in the previous identity we get
\[\begin{split}
\int_{A} \psi(a) \Big(\int_0^T \int_{-R}^R (\partial_t \phi(t,x)+\partial_x \phi(t,x) v(t)^a(x) ) \mbox{d} \nu(t)^a(x) \mbox{d} t\Big) \mbox{d} \mu(a)\\
= \int_{A} \psi(a) \Big( \int_{-R}^R \phi(T,x) \mbox{d} \nu(T)^a(x)-  \int_{\R} \phi(0,x) \mbox{d} \nu_0^a(x)  \Big) \mbox{d} \mu(a).
\end{split}\]
In other words, we have proved the following:
\begin{lem}\label{contok}
For $\mu$-almost every $a$, the limit curve $t\mapsto \nu(t)^a$ solves the continuity equation \pref{s1} associated to the limit velocity $t\mapsto v(t)^a$. 
\end{lem}

It remains to check that 

\begin{lem}\label{sdok}
For a.e. $t\in [0,T]$, we have $\bv(t)\in -\partial J(\bn(t))$. 
\end{lem}

\begin{proof}
By construction of the curves $\bv_\tau$ and $\bn_\tau$ and thanks to Lemma \ref{lem-el}, we have seen in \pref{sousgradtau} that
\[ \bv_\tau(t)\in  -\partial J(\bn_\tau(t)), \; \forall t\in [0,T]\]
which means that for every $\tau>0$, every $t\in [0,T]$ and every $\bet\in T(\bn_\tau(t))$ (as defined in Remark \ref{sousdiffdes}), we have  
\begin{equation}\label{sdapprox}
J({\bn_\tau(t)}_\bet)-J(\bn_\tau(t))\ge -\int_{A\times \R^2} v_\tau^a(t)(y)(z-y) \mbox{d} \eta^{a, y} (z) \mbox{d} \nu_\tau(t)^a(y) \mbox{d} \mu(a).
\end{equation}
We wish to prove that there exists $S\subset [0,T]$, $\lambda$-negligible, such that for every $t\in [0,T]\setminus S$ and every $\eta\in T(\bn(t))$, one has
\begin{equation}\label{sdlim}
J({\bn(t)}_\bet)-J(\bn(t))\ge -\int_{A\times \R^2} v^a(t)(y)(z-y) \mbox{d} \eta^{a, y} (z) \mbox{d} \nu(t)^a(y) \mbox{d} \mu(a).
\end{equation}
To pass to the limit $\tau=\tau_n$, $n\to \infty$ in \pref{sdapprox} to obtain \pref{sdlim}, we shall proceed in several steps. Let us remark  that it is enough to prove \pref{sdapprox} when  $\eta^{a,y}$ is supported by a fixed compact interval $[-R', R']$ (and then to take an exhaustive sequence of such compact intervals).  Let us also recall that, thanks to Lemma \ref{basic} and \pref{estim4},  $J(\bn_{\tau_n}(t))$ converges to $J(\bn(t))$ as $n\to \infty$ uniformly on $[0,T]$.

{\bf{Step 1:}} Let us first consider the case where $\bet$ is continuous in the sense that $(a,y)\in K \mapsto \int_{[-R', R']} \varphi(z) \mbox{d} \eta^{a,y}(z)$ is continuous for every $\varphi \in C(\R)$. Let $\phi\in C(A\times \R)$. Since $\varphi_\bet$ defined by $\varphi_\bet(a,y):=\int \phi(a,z) \mbox{d} \eta^{a,y}(z)$ belongs to $C(K)$, using the fact that 
\[\begin{split}
\langle \phi, \bn_{\tau_n}(t)_\bet \otimes \mu\rangle=\langle \varphi_\bet, \bn_{\tau_n}(t) \otimes \mu\rangle,\\
\langle \phi, \bn(t)_\bet \otimes \mu\rangle=\langle \varphi_\bet, \bn(t) \otimes \mu\rangle
\end{split}\]
 and \pref{estim4}, we deduce that $\lim_n d_w(\bn_{\tau_n}(t)_\bet, \bn(t)_\bet)=0$ for every $t\in [0,T]$. Hence,  thanks to Lemma \ref{basic}, we have 
 \begin{equation}\label{sdapprox1}
 \lim_{n}  [J({\bn_{\tau_n}(t)}_\bet)-J(\bn_{\tau_n}(t))]= J({\bn_\tau(t)}_\bet)-J(\bn_\tau(t)), \; \forall t\in [0,T].
 \end{equation}
Let $\varphi \in C([0,T])$, $\varphi\ge 0$. Using \pref{sdapprox} gives  
\[\begin{split}
&\int_0^T \varphi(t)   [J({\bn_{\tau_n}(t)}_\bet)-J(\bn_{\tau_n}(t))]\mbox{d} t \ge\\
& -\int_{[0,T]\times A\times \R^2}  \varphi(t) v_{\tau_n}^a(t)(y)(z-y) \mbox{d} \eta^{a, y} (z) \mbox{d} \nu_{\tau_n}(t)^a(y) \mbox{d} \mu(a)   \mbox{d} t\\
&=-\int_{[0,T]\times K} \varphi(t) \psi(a,y) \mbox{d} q_{\tau_n}(t)^a(y) \mbox{d}\mu(a) \mbox{d} t
\end{split}\]
where  
\[\psi(a,y):=\int (z-y)\mbox{d} \eta^{a,y}(z) \]
belongs to $C(K)$. We then deduce from \pref{cvq}, \pref{sdapprox1} and Lebesgue's dominated convergence that  
\[\begin{split}
&\int_0^T \varphi(t)   [J({\bn(t)}_\bet)-J(\bn(t))]\mbox{d} t \ge\\
&=-\int_{[0,T]\times K} \varphi(t) \psi(a,y) \mbox{d} q^a(y) \mbox{d}\mu(a) \mbox{d} t\\
&= -\int_{[0,T] \times A\times \R^2}     \varphi(t) v^a(t)(y)(z-y) \mbox{d} \eta^{a, y} (z) \mbox{d} \nu(t)^a(y) \mbox{d} \mu(a)   \mbox{d} t.
\end{split}\]
This implies that there exists a negligible subset $S_\bet$ of $[0,T]$ outside which \pref{sdlim} holds. 
\smallskip

{\bf{Step 2:}} For every $N\in \N^*$, let $\Delta_N:=\{(\alpha_0, \cdots, \alpha_{2N-1}) \in \R_+^{2N} \; : \; \sum_{k=0}^{2N-1} \alpha_i=1\}$, $F_N$ be a countable and dense family in $C(K, \Delta_N)$, and consider 
\[D_N:=\{ (a,y)\in K \mapsto \sum_{k=0}^{2N-1} \alpha_k(a,y) \delta_{z_k^N}, \; (\alpha_0, \ldots, \alpha_{2N-1})\in F_N\}, \;  D:=\bigcup_{N\in \N^*} D_N\]
where for $k=0, \ldots, 2N-1$, $z_k^N$ denotes the midpoint of the interval $[-R'+kR'/N, -R'+(k+1)R'/N]$. Since $D$ is countable and its elements belong to $C(K, (\PP([-R', R']), W_2))$, it follows from Step 1, that  \pref{sdlim} holds for every  $\bet\in D$ and every $t\in [0,T]\setminus S$ where $S$ is the $\lambda$-negligible set
\begin{equation}\label{defdus}
S:=\bigcup_{\bet\in D} S_{\bet}.
\end{equation}

\smallskip

{\bf{Step 3:}} Let $t\in [0,T]\setminus S$,  and  $\bet\in T(\bn)$ having its support in $[-R', R']$. Note that now we are working with a fixed $t$ so that we just have to suitably approximate $\bet$ by a sequence in $D$. For $N\in \N^*$, first define for every $(a,y)\in K$ the discrete measure
\begin{equation}\label{approxeta1}
\sum_{k=0}^{2N-1} f_k^N(a,y) \delta_{z_k^N}, \; f_k^N(a,y) :=\eta^{a,y}(I_k^N)
\end{equation}
where $I_k^N$ is the interval $[-R'+kR'/N, -R'+(k+1)R'/N)$ if $k=0, \ldots, 2N-2$ and $I_{2N-1}^N:=[R'(1-1/N), R']$. We then have
\begin{equation}\label{approxteta2}
\sup_{(a,y)\in K} W_1\Big(\eta^{a,y}, \sum_{k=0}^{2N-1} f_k^N(a,y) \delta_{z_k^N}\Big) \le \frac{R'}{N}.
\end{equation}
The function $(f_k^N)_{k=0, \ldots, 2N-1}$ is not continuous but belongs to $L^1(\bn(t)\otimes \mu, \Delta_N)$. Since $C(K, \Delta_N)$ is dense in $L^1(\bn(t)\otimes \mu, \Delta_N)$, there exist $(g_0^N, \ldots, g_{2N-1}^N)\in C(K, \Delta_N)$ such that
\begin{equation}\label{approxteta3}
\sum_{k=0}^{2N-1} \int_K \vert f_k^N(a,y)-g_k^N(a,y) \vert \mbox{d} \nu(t)^a(x) \mbox{d} \mu(a) \le \frac{1}{N}. 
\end{equation}
Since we have chosen $F_N$ dense in $C(K, \Delta_N)$, there exist $\alpha=(\alpha_0^N, \ldots, \alpha_{2N-1}^N)\in F_N$ such that 
\begin{equation}\label{approxteta4}
\sum_{k=0}^{2N-1} \sup_{(a,y)\in K} \vert g_k^N(a,y)-\alpha_k^N(a,y) \vert \le \frac{1}{N}. 
\end{equation}
We then define $\eta_N\in D$ by
\[\eta_N^{a,y}:=\sum_{k=0}^{2N-1} \alpha_k^N(a,y) \delta_{z_k^N}.\]
Thanks to Kantorovich duality formula \pref{kantodual}, it is easy to  see that for every $\alpha$ and $\beta$ in $\Delta_N$, $W_1(\sum_k \alpha_k \delta_{z_k^N}, \sum_k \beta_k \delta_{z_k^N}) \le R' \sum_k \vert \alpha_k-\beta_k \vert$. In particular, thanks to \pref{approxteta3}, we have
\begin{equation}\label{approxeta33}
\int_K W_1\Big(\sum_k f_k^N (a,y) \delta_{z_k^N}, \sum_k g_k^N (a,y) \delta_{z_k^N}\Big) \mbox{ d} (\bn(t)\otimes \mu)(a,y)  \le \frac{R'}{N}.
 \end{equation}

Similarly,  \pref{approxteta4} implies that
 \begin{equation}\label{approxteta5}
\sup_{(a,y)\in K} W_1\Big(\eta_N^{a,y},  \sum_{k=0}^{2N-1}  g_k^N(a,y)\delta_{z_k^N}\Big)  \le \frac{R'}{N}. 
\end{equation}
 We know, from Step 2 that for every $N\in \N^*$:
 \begin{equation}\label{approxteta6}
 J({\bn(t)}_{\bet_N})-J(\bn(t))\ge -\int_{A\times \R^2} v^a(t)(y)(z-y) \mbox{d} \eta_N^{a, y} (z) \mbox{d} \nu(t)^a(y) \mbox{d} \mu(a). 
\end{equation}

Thanks to \pref{approxteta2}, \pref{approxeta33}, \pref{approxteta5} and the triangle inequality, we  have
\begin{equation}\label{approxteta7}
\lim_{N\to \infty} \int_K W_1(\eta^{a,y}, \eta_N^{a,y}) \mbox{ d} (\bn(t)\otimes \mu)(a,y)=0.
\end{equation}
Recalling that $\bv(t)\in L^{\infty}(\bn(t)\otimes \mu)$ and using \pref{rappel2}, we have 
\[\begin{split}\Big\vert   \int_{K} v^a(t)(y) \Big( \int_{[-R',R']} (z-y) \mbox{d} (\eta_N^{a, y}-\eta^{a,y}) (z)\Big) \mbox{d} \nu(t)^a(y) \mbox{d} \mu(a)     \Big\vert  \\
 \le \Vert \bv\Vert_{L^{\infty}(\bn(t)\otimes \mu)} \int_K W_1(\eta^{a,y}, \eta_N^{a,y}) \mbox{ d} (\bn(t)\otimes \mu)(a,y)
\end{split}\]
so that the right-hand side of \pref{approxteta6} converges to 
\[-\int_{A\times \R^2} v^a(t)(y)(z-y) \mbox{d} \eta^{a, y} (z) \mbox{d} \nu(t)^a(y) \mbox{d} \mu(a)\]
as $N\to \infty$. As for the convergence of the right-hand side of \pref{approxteta6}, we have to show that $\lim_N W_1(\bn_{\bet_N}\otimes \mu, \bn_{\bet}\otimes \mu)=0$. For this, we shall use the Kantorovich-duality formula \pref{kantodual} and observe that if $\phi\in C(K)$ is $1$-Lipschitz then
\[\int_K \phi(a,y) \mbox{d}  ((\bn_{\bet_N}- \bn_{\bet})\otimes \mu)  (a,y) \le  \int_K W_1(\eta^{a,y}, \eta_N^{a,y}) \mbox{ d} (\bn(t)\otimes \mu)(a,y)\]
which tends to $0$ as $N\to \infty$ thanks to \pref{approxteta7}. Using Lemma \ref{basic} we then have $\lim_{N\to \infty} J({\bn(t)}_{\bet_N})= J({\bn(t)}_{\bet})$. Passing to the limit $N\to \infty$ in \pref{approxteta6} gives the desired inequality \pref{sdlim}. This shows that $\bv(t)\in -\partial J(\bn(t))$ for every $ t\in [0,T]\setminus S$. 
\end{proof}

We deduce from Lemma \ref{contok} and Lemma \ref{sdok} the following existence result:

\begin{thm}\label{refexist}
If \pref{hypf0} holds, then for any $T>0$, there exists a gradient flow of $J$ starting from $\bn_0$ on the time interval $[0,T]$. In particular, there exists measure solutions to the system \pref{s1}-\pref{s2}-\pref{s3}.
\end{thm}

\section{Uniqueness and concluding remarks}\label{sec-wp}

\subsection{Uniqueness and stability}

Thanks to \pref{monsousdiff}, we easily deduce uniqueness and stability:

\begin{thm}\label{well-posed}
Let $\bn_0$ and $\bt_0$ be in $X_R$. If $t\mapsto \bn(t)$ and $t\mapsto \bt(t)$ are gradient flows of $J$ starting respectively from $\bn_0$ and $\bt_0$, then 
\[d(\bn(t), \bt(t))\le d(\bn_0, \bt_0), \; \forall t\in \R_+.\]
In particular there is a unique gradient flow of $J$ starting from $\bn_0$. 
\end{thm}

\begin{proof}
By definition there exists  velocity fields $\bv$ and $\bw$ such that for a.e. $t$, $\bv(t)=(v(t)^a)_{a\in A}\in -\partial J(\bn(t))$ and $\bw(t)=(w(t)^a)_{a\in A}\in -\partial J(\bt(t))$ and for $\mu$-almost every $a$, one has
\begin{equation}
\partial_t \nu^a+ \partial_x (\nu^a v^a)=\partial_t \theta^a+ \partial_x (\theta^a w^a)=0, \; \nu^a\vert_{t=0}=\nu_0^a, \; \theta^a\vert_{t=0}=\theta_0^a.
\end{equation}
Since $v^a$ and $w^a$ are bounded in $L^{\infty}(\nu^a)$ and $L^{\infty}(\theta^a)$ respectively, it follows from well-known arguments (see \cite{AGS}, in particular Theorem 8.4.7 and Lemma 4.3.4) that $t\mapsto W_2^2(\nu_t^a, \theta_t^a)$ is a Lipschitz function and that for any family of optimal plans $\gamma_s^a$ between $\nu_s^a$ and $\theta_s^a$  for $t_1\le t_2$ one has:
\[W_2^2(\nu_{t_2}^a, \theta_{t_2}^a)\le  W_2^2(\nu_{t_1}^a, \theta_{t_1}^a) + \int_{t_1}^{t_2} \Big( \int_{\R^2} (v^a(s)(y)-w^a(s)(z))(y-z) \mbox{d} \gamma^a_s(y,z)  \Big) \mbox{d} s.\]
Integrating the previous inequality  gives
\[d^2(\nu_{t_2}, \theta_{t_2})\le  d^2(\nu_{t_1}, \theta_{t_1}) + \int_{t_1}^{t_2} \Big( \int_{A\times \R^2} (v^a(s)(y)-w^a(s)(z))(y-z) \mbox{d} \gamma^a_s(y,z)  \mbox{d} \mu(a)\Big) \mbox{d} s.\]
But since $\bv(s)\in -\partial J(\bn(s))$ and $\bw(s)\in -\partial J(\bt(s))$ for a.e. $s$, the monotonicity relation \pref{monsousdiff} gives 
 \[\int_{A\times \R^2} (v^a(s)(y)-w^a(s)(z))(y-z) \mbox{d} \gamma^a_s(y,z)  \mbox{d} \mu(a)\le 0.\]
 We then obtain the desired contraction estimate.

\end{proof}

\subsection{Concluding remarks}

\subsection*{More general initial conditions}

We would like to mention here that in our main results of existence and uniqueness of a gradient flow for $J$, the assumption that $\rho_0$ is atomless plays no significant role. Actually, our results hold for any compactly supported initial condition $\bn_0$ (we did not investigate the extension to the case where this assumption is relaxed to a second moment bound, but this is probably doable). The assumption that $\rho_0$ is atomless was used only to select unambiguously the Cauchy datum $\nu_0^a$ in order to justify the reformulation of the initial kinetic equation by taking advantage of the first integral trick of section \ref{sec-firstint}. We suspect that in the case where $\rho_0$ is a discrete measure, there might be an interesting connection between gradient flows solutions and some solutions of the  initial ODE system \pref{ode} but a more precise investigation is left for the future.

\subsection*{Higher dimensions, more general functionals}

The motivation for the present work comes from kinetic models of granular media. Since the first integral trick of section \ref{sec-firstint} is very specific to the quadratic interaction kernel case in dimension one, all our subsequent analysis has been performed in dimension one only. However, it is obvious (but we are not aware of any practical examples in kinetic theory) that our arguments can be used also to study systems of continuity equations in $\R^d$ for infinitely many species (labeled by a parameter $a$) such as 
\[\partial_t \nu^a +\dive_x\Big (\nu^a (\nabla_x V(a,x) + \int_{A\times \R^d}  \nabla_x W(a,b, x,y) \mbox{d} \nu^b(y)  \mbox{d} \mu(b) )\Big)=0,\]
which (taking for instance $W$ symmetric $W(a,b,x,y)=W(b,a,y,x)$), can be seen as the gradient flow of
\[J(\bn):=\int_{A\times \R^d} V \mbox{d} (\bn\otimes \mu)+ \frac{1}{2} \int_{A\times \R^d} \int_{A\times \R^d}  W  \mbox{d} (\bn\otimes \mu) \otimes \mbox{d} (\bn\otimes \mu).\]




{\bf{Acknowledgements:}}  The authors are grateful to Yann Brenier, Reinhard Illner and Maxime Laborde for fruitful discussions about this work. M.A. acknowledges the support of NSERC through a Discovery Grant. G.C. gratefully acknowledges the hospitality of the Mathematics and Statistics Department at UVIC (Victoria, Canada), and the
support from the CNRS, from the ANR, through the project ISOTACE (ANR-12-
MONU-0013) and from INRIA through the \emph{action exploratoire} MOKAPLAN.

\end{document}